\documentclass[12pt]{article}
\usepackage{amsmath,amsthm,amssymb,euscript,verbatim}
\newtheorem{theorem}{Theorem}[section]

\newtheorem{lemma}{Lemma}[section]
\newtheorem{corollary}{Corollary}[section]
\newtheorem{remark}{Remark}[section]

\newtheorem{definition}{Definition}[section]
\begin{document}
\title
{\bf Notes on the Invariance of Tautness Under Lie Sphere Transformations}
\author
{Thomas E. Cecil}
\maketitle

\begin{abstract}
An embedding $\phi:V \rightarrow S^n$ 
of a compact, connected manifold $V$ into the unit sphere $S^n \subset {\bf R}^{n+1}$ 
is said to be taut, if every nondegenerate spherical distance function $d_p$, $p \in S^n$, is a perfect
Morse function on $V$, i.e., it 
has the minimum
number of critical points  on $V$ required by the Morse inequalities.
In these notes, we give an exposition of the
proof of the invariance of tautness under Lie sphere transformations due to \'{A}lvarez Paiva \cite{Alv}. First we extend
the definition of tautness of submanifolds of $S^n$ to 
the concept of Lie-tautness of Legendre submanifolds of the contact manifold 
$\Lambda^{2n-1}$ of projective lines on the Lie quadric $Q^{n+1}$.
This definition has the property that if $\phi:V \rightarrow S^n$ is an embedding of a compact, connected manifold $V$, then
$\phi(V)$ is a taut submanifold in $S^n$ if and only if the Legendre lift $\lambda$
of $\phi$ is Lie-taut. Furthermore, Lie-tautness is invariant under the action of Lie sphere transformations on Legendre submanifolds.  
As a consequence, we get that
if $\phi:V \rightarrow S^n$ and $\psi:V \rightarrow S^n$ are two embeddings of a compact, connected  
manifold $V$ into $S^n$, 
such that their corresponding Legendre lifts are related by a Lie sphere transformation,
then $\phi$ is a taut embedding if and only if $\psi$ is a taut embedding.  Thus, in that sense, tautness is invariant under Lie sphere transformations.
The key idea is to formulate tautness in terms of real-valued
functions on $S^n$ whose level sets form a parabolic pencil 
of unoriented spheres in $S^n$, and then show that this is equivalent to the usual formulation 
of tautness in terms of spherical distance functions,
whose level sets in $S^n$ form a pencil of unoriented concentric spheres.  
\end{abstract}

\section{Introduction}
\label{intro}
An embedding $\phi:V \rightarrow S^n$ 
of a compact, connected manifold $V$ into the unit sphere $S^n \subset {\bf R}^{n+1}$ 
is said to be {\em taut}, if every nondegenerate spherical distance function $d_p$, $p \in S^n$, is a perfect
Morse function on $V$, i.e., it has the minimum
number of critical points  on $V$ required by the Morse inequalities. (See Section \ref{sec:1} for more detail.)

In these notes, we give an exposition of the proof of the invariance of tautness under Lie sphere transformations
due to \'{A}lvarez Paiva \cite{Alv}. These notes follow Section 4.6 of the author's book \cite[pp. 82-95]{Cec1} very closely, and much of the text is taken directly from that book. (See also Section 4.6 the book of 
Cecil-Ryan \cite[pp. 224--231]{CR8}
for a shorter version of the proof, which omits some of the detailed calculations.)

To formulate this result properly, it is necessary to consider submanifold theory in the context of Lie sphere geometry, since Lie sphere transformations act on the space of oriented hyperspheres in $S^n$, and not on the space $S^n$ itself.  
(See Section \ref{sec:3} for more detail.)

Following  \'{A}lvarez Paiva \cite{Alv}, we extend
the definition of tautness of submanifolds of $S^n$ to 
the concept of Lie-tautness of Legendre submanifolds of the contact manifold 
$\Lambda^{2n-1}$ of projective lines on the Lie quadric $Q^{n+1}$.
 (See Section \ref{sec:4} for more detail.)

We show that this definition of Lie-tautness has the property that if $\phi:V \rightarrow S^n$ is an embedding of a compact, connected manifold $V$, then $\phi(V)$ is a taut submanifold of $S^n$ if and only if the Legendre lift $\lambda$
of $\phi$ is Lie-taut. Furthermore, we show that
this notion of Lie-tautness is invariant under the action of Lie sphere transformations on Legendre submanifolds.  

As a consequence, we get that
if $\phi:V \rightarrow S^n$ and $\psi:V \rightarrow S^n$ are two embeddings of a compact, connected  
manifold $V$ into $S^n$, 
such that their corresponding Legendre lifts are related by a Lie sphere transformation,
then $\phi$ is a taut embedding if and only if $\psi$ is a taut embedding.  Thus, in that sense, tautness is invariant under Lie sphere transformations.

The key idea is to formulate tautness in terms of real-valued
functions on $S^n$ whose level sets form a parabolic pencil 
of unoriented spheres in $S^n$, and then show that this is equivalent to the usual formulation 
of tautness in terms of spherical distance functions, whose level sets in $S^n$ form a pencil of unoriented concentric spheres.

\section{Taut Submanifolds of Real Space Forms}
\label{sec:1}

Let $\phi:V \rightarrow {\bf R}^n$ be an immersion of a compact, connected manifold $V$ into ${\bf R}^n$
with $\dim V < n$.  For
$p \in {\bf R}^n$, the {\em Euclidean distance function}, $L_p:V \rightarrow {\bf R}$, is defined by the formula
\begin{displaymath}
L_p(x) = |p - \phi(x)|^2.
\end{displaymath}
If $p$ is not a focal point of the submanifold $\phi$, then $L_p$ is a
{\em nondegenerate function}
(i.e., a {\em Morse function}), 
that is, all of its critical points are nondegenerate
(see, for example, Milnor \cite[pp. 32--38]{Mil}).  By the Morse inequalities, 
the number $\mu (L_p)$ of critical points of a
nondegenerate distance function on $V$ satisfies 
\begin{displaymath}
\mu (L_p) \geq \beta (V;{\bf F}), 
\end{displaymath}
the sum of the 
${\bf F}$-Betti numbers 
of $V$ for any field ${\bf F}$.  The immersion $\phi$ is said to be {\em taut} if there exists a field ${\bf F}$
such that every nondegenerate Euclidean distance function has $\beta (V;{\bf F})$ critical points on $V$.  If it is
necessary to distinguish the field ${\bf F}$, we will say that $\phi$ is ${\bf F}$-{\em taut}.  The field
${\bf F} = {\bf Z}_2$ has been satisfactory for most considerations involving tautness
so far, and we will use ${\bf Z}_2$ as the
field in these notes.

Tautness was first studied by Banchoff \cite{Ban1}, who determined all taut 
2-dimensional surfaces in Euclidean space.  Carter and West
\cite{CW1} introduced the term ``taut immersion,'' and
proved many basic results in the field.  Among these is the fact that a taut immersion
must be an embedding, 
since if $p \in {\bf R}^n$ were a double point, then the function $L_p$ would have two
absolute minima instead of just one, as required by tautness.  (See also the books \cite{CR7}, \cite{CR8},
\cite{CC3}, for more on tautness.)

Carter and West \cite{CW1} also showed that tautness is invariant under M\"{o}bius (conformal) transformations 
of ${\bf R}^n \cup \{\infty \}$.  Further, an embedding
$\phi:V \rightarrow {\bf R}^n$ of a compact manifold $V$ is taut if and only if the embedding
$\sigma \phi:V \rightarrow S^n$ has the property that every nondegenerate 
spherical distance function 
\begin{equation}
\label{eq:spherical-dist-fn}
d_p (q) = \cos^{-1} (p \cdot q)
\end{equation}
has $\beta (V;{\bf F})$ critical points on $V$.  Here $S^n$ is the unit sphere centered at the origin in Euclidean space
${\bf R}^{n+1}$, and $\sigma: {\bf R}^n \rightarrow S^n - \{P\}$ is the inverse map of 
stereographic projection $\tau:S^n - \{P\} \rightarrow  {\bf R}^n$
with pole $P \in S^n$. The term $p \cdot q$ denotes the usual Euclidean inner product on 
${\bf R}^{n+1}$.

Since a spherical distance function $d_p$ is closely related to a Euclidean height function
$\ell_p (q) = p \cdot q$, for $p,q \in S^n$, one can easily show that the embedding $\phi$ 
is taut if and only if the spherical embedding $\sigma \phi$ is {\em tight}, 
i.e., every nondegenerate height function $\ell_p$ has $\beta (V;{\bf F})$ critical points
on $V$. It is often simpler to use height functions rather than spherical distance functions when studying tautness
for submanifolds of $S^n$. 

In the proof of the Lie invariance of tautness due to \'{A}lvarez Paiva \cite{Alv},
it is more convenient to consider embeddings into $S^n$ rather
than ${\bf R}^n$.  The remarks in the two paragraphs above show that the two theories are essentially equivalent for
embeddings of compact manifolds $V$ via stereographic projection.

Kuiper \cite{Ku1}--\cite{Ku4} gave a reformulation of critical point theory in terms of an 
injectivity condition on homology 
which has turned out to be very useful in the theory of tight and taut immersions.
Let $f$ be a nondegenerate function on a manifold $V$.  We define the {\em sublevel set}
\begin{equation}
\label{eq:3.6.1}
V_r(f) = \{x \in V \mid f(x) \leq r\}, \quad r \in {\bf R}. 
\end{equation}
The next theorem, which follows immediately from Theorem 2.2 of Morse--Cairns \cite[p. 260]{MC} 
(see also Theorem 2.1 of Chapter 1 of Cecil--Ryan \cite{CR7}), was a key to Kuiper's formulation of these conditions.

\begin{theorem}
\label{thm:3.6.1} 
Let $f$ be a nondegenerate function on a compact, connected manifold $V$.  For a given field ${\bf F}$, the number
$\mu (f)$ of critical points of $f$ equals the sum $\beta (V;{\bf F})$ of the ${\bf F}$-Betti numbers of $V$ if and
only if the map on homology,
\begin{equation}
\label{eq:3.6.2}
H_*(V_r(f);{\bf F}) \rightarrow H_*(V;{\bf F}),
\end{equation}
induced by the inclusion $V_r(f) \subset V$ is injective for all $r \in {\bf R}$.
\end{theorem}

Of course, for an embedding $\phi:V \rightarrow S^n$ and a height function 
$\ell_p$, the set 
$V_r(\ell_p)$, is equal to $\phi^{-1}(B)$, where $B$ is the closed ball in $S^n$ obtained by intersecting
$S^n$ with the half-space in ${\bf R}^{n+1}$ determined by the inequality $\ell_p (q) \leq r$.
Kuiper \cite{Ku4} used the 
continuity property of ${\bf Z}_2$-\v{C}ech homology 
to formulate tautness in terms
of $\phi^{-1}(B)$, for all closed balls $B$ in $S^n$, not just those centered at non-focal points of $\phi$.  
Thus, Kuiper proved
the following theorem (see also Theorem 1.12 of Chapter 2 of Cecil--Ryan \cite[p. 118]{CR7}).    
\begin{theorem}
\label{thm:3.6.2} 
Let $\phi:V \rightarrow S^n$ be an embedding of a compact, connected manifold $V$ into $S^n$. Then
$\phi$ is ${\bf Z}_2$-taut if and only if for every closed ball $B$ in $S^n$, the induced homomorphism
$H_*(\phi^{-1}(B)) \rightarrow H_*(V)$ in ${\bf Z}_2$-\v{C}ech homology is injective.
\end{theorem}

In the 1980's,
Pinkall \cite{P3}--\cite{P4} situated the study of submanifolds of $S^n$ or ${\bf R}^n$ within the context of Lie sphere 
geometry.  
This approach has proven to be very useful in subsequent research on 
taut and Dupin submanifolds of $S^n$ or ${\bf R}^n$.  In Section \ref{sec:3}, we give a brief description of this approach to submanifold theory. (See 
also the book \cite{Cec1} for more detail.)

In Section \ref{sec:4}, we present the details of the proof of \'{A}lvarez Paiva \cite{Alv} of the invariance of tautness under Lie sphere transformations. The exposition in these notes follows Section 4.6 of the author's book \cite[pp. 82-95]{Cec1} very closely, and much of the text is taken directly from that book.  (See also the book 
\cite[pp. 224--231]{CR8} by Cecil and Ryan,
for a shorter version of the proof which omits some of the detailed calculations.)

For a generalization of tautness to
submanifolds of arbitrary Riemannian manifolds, see the paper of Terng and Thorbergsson \cite{TTh1}, and for a generalization of tautness to symplectic geometry, see the paper of \'{A}lvarez Paiva \cite{Alv-99}.

\section{Submanifolds in Lie Sphere Geometry}
\label{sec:3}
We now give a brief description of the method for studying submanifolds
in $S^n$ or ${\bf R}^n$ within the context of Lie sphere geometry
(introduced by Lie \cite{Lie}, see also Blaschke \cite{Bl}).  The reader is referred to the papers of
Pinkall \cite{P4}, Chern \cite{Chern}, Cecil and Chern \cite{CC1}--\cite{CC2},
Cecil-Chi-Jensen \cite{CCJ2},  Miyaoka \cite{Mi1}--\cite{Mi4},
or the books of Cecil \cite{Cec1}, or
Jensen, Musso and Nicolodi \cite{Jensen-Musso-Nicolodi}, for more detail.

Lie sphere geometry is situated in real projective space ${\bf P}^n$,
so we now briefly review some concepts and notation from projective geometry.
We define an equivalence relation on ${\bf R}^{n+1} - \{0\}$ by setting
$x \simeq y$ if $x = ty$ for some nonzero real number $t$.  We denote
the equivalence class determined by a vector $x$ by $[x]$.  Projective
space ${\bf P}^n$ is the set of such equivalence classes, and it can
naturally be identified with the space of all lines through the origin
in ${\bf R}^{n+1}$.  The rectangular coordinates $(x_1, \ldots, x_{n+1})$
are called {\em homogeneous coordinates}
of the point $[x]$ in ${\bf P}^n$, and they
are only determined up to a nonzero scalar multiple.  

In these notes, we are working primarily with oriented hyperspheres in the unit sphere $S^n \subset {\bf R}^{n+1}$,
rather than in ${\bf R}^n$. 
So we will describe that case here.  For more detail on case of oriented hyperspheres in ${\bf R}^n$,
see \cite[pp. 14--23]{Cec1}).

A {\em Lie sphere} in $S^n$ is an oriented hypersphere or a point sphere in $S^n$.
The set of all Lie spheres
is in bijective correspondence with the set of all points $[x] = [(x_1,\ldots,x_{n+3})]$
in projective space ${\bf P}^{n+2}$ that lie
on the quadric hypersurface $Q^{n+1}$ 
determined by the equation 
$\langle x, x \rangle = 0$, where
\begin{equation}
\label{Lie-metric}
\langle x, y \rangle = -x_1 y_1 + x_2 y_2 + \cdots +x_{n+2} y_{n+2} - x_{n+3} y_{n+3}
\end{equation}
is a bilinear form of signature $(n+1,2)$ on ${\bf R}^{n+3}$.  
We let  ${\bf R}^{n+3}_2$ denote ${\bf R}^{n+3}$ endowed with
the indefinite inner product (\ref{Lie-metric}).  The quadric 
$Q^{n+1} \subset {\bf P}^{n+2}$  is called the {\em Lie quadric}. A vector $x$ in ${\bf R}^{n+3}_2$ is called {\em spacelike},
{\em timelike}, or {\em lightlike}, respectively, depending on whether $\langle x, x \rangle$ is positive, negative or zero.

The specific correspondence is as follows (see \cite[pp. 16--18]{Cec1}).  We identify $S^n$ 
with the unit sphere in ${\bf R}^{n+1} \subset {\bf R}^{n+3}_2$, where ${\bf R}^{n+1}$ is spanned by the 
standard basis vectors $\{e_2,\ldots,e_{n+2}\}$ in ${\bf R}^{n+3}_2$.
Then the oriented hypersphere with center $p \in S^n$ and
signed radius $\rho$ corresponds to the point in $Q^{n+1}$ with homogeneous coordinates,
\begin{equation}
\label{eq:1.4.4}
(\cos \rho , p, \sin \rho ),
\end{equation}
where $ - \pi < \rho < \pi$.

We can designate the orientation of the sphere by the sign of the radius
$\rho$ as follows.  A positive radius $\rho$ in 
(\ref{eq:1.4.4}) corresponds to the orientation of the sphere given by 
the field of unit normals which are tangent vectors to geodesics in $S^n$ 
going from $-p$ to $p$, and a negative radius corresponds
to the opposite orientation.
Each oriented sphere can be considered in two ways, with center $p$ and signed radius $\rho, - \pi < \rho < \pi$,
or with center $-p$ and the appropriate signed radius $\rho \pm \pi$.
Point spheres $p$ in $S^n$ correspond to those points $[(1, p, 0)]$ in $Q^{n+1}$ with radius $\rho = 0$, and they do not have an orientation.

Due to the signature of the inner product $\langle \ ,\ \rangle$, the
Lie quadric $Q^{n+1}$ contains projective lines but no linear
subspaces of ${\bf P}^{n+2}$ of higher dimension (see, for example, \cite[p. 21]{Cec1}).  
A straightforward calculation
shows that if $[x]$ and $[y]$ are two points on the quadric, 
then the line $[x,y]$ 
lies entirely on $Q^{n+1}$ if and only if the inner product $\langle x,y \rangle=0$.  Geometrically,
this condition means that the hyperspheres in $S^n$ corresponding
to $[x]$ and $[y]$ are in {\em oriented contact}, i.e., they are tangent to each other
and have the same orientation at the point of contact.
For an oriented sphere and a point sphere, oriented contact means that the point lies on the sphere.
The 1-parameter family
of Lie spheres  in $S^n$ corresponding to the points on a line $\ell$ lying on the Lie quadric is called a 
{\em parabolic pencil of spheres}, which we will now describe further (see \cite[pp. 19--23]{Cec1}).

Using linear algebra on the indefinite inner product space ${\bf R}^{n+3}_2$, one can prove the following result
(see \cite[p. 22]{Cec1}).
\begin{theorem}
\label{cor:1.5.5} 
Let $[z]$ be a timelike point in ${\bf P}^{n+2}$, and let $\ell$ a line that lies on $Q^{n+1}$.  Then $\ell$
intersects the space $z^\perp$ in exactly one point.
\end{theorem}

Suppose now that $\ell$ is a line lying on the
quadric.  Using Theorem \ref{cor:1.5.5} and equation \eqref{eq:1.4.4}, we see that $\ell$ intersects both
$e_1^\perp$ and $e_{n+3}^\perp$ in exactly one point.  So the corresponding 
parabolic pencil of oriented spheres contains exactly
one point sphere  $[k_1]$ and one great sphere $[k_2]$, represented respectively by the points,
\begin{equation}
\label{eq:point-sphere-great-sphere} 
[k_1] = [(1,p,0)], \quad [k_2] = [(0,\xi,1)],
\end{equation}
on $Q^{n+1}$.  The condition $\langle k_1, k_2 \rangle = 0$ is equivalent to the condition $p \cdot \xi = 0$, i.e., $\xi$
is a unit vector tangent to $S^n$ at $p$.  Hence, the parabolic pencil of spheres corresponding to $\ell$ can
be identified with the point $(p, \xi)$ in  the unit tangent bundle $T_1S^n$ of $S^n$.  

Thus, we have a  a bijective correspondence between the points of $T_1S^n$ and the
manifold  $\Lambda^{2n-1}$ of projective lines on $Q^{n+1}$ given by the map
\begin{equation}
\label{eq:3.1.8}
(p, \xi) \mapsto [k_1(p, \xi), k_2 (p, \xi)],
\end{equation}
where
\begin{equation}
\label{eq:3.1.9}
k_1(p, \xi) = (1,p,0), \quad k_2(p, \xi) = (0,\xi,1).
\end{equation}
Since every line on the quadric contains exactly one point sphere and one great sphere,
the correspondence in \eqref{eq:3.1.8} is bijective.  We use the
correspondence \eqref{eq:3.1.8} to put a differentiable structure on the manifold $\Lambda^{2n-1}$ of projective 
lines that lie on the Lie quadric $Q^{n+1}$
in such a way that
the map in \eqref{eq:3.1.8} becomes a diffeomorphism.
We will refer to an element $(p, \xi)$ in  $T_1S^n$ as a {\em contact element}.

The points on the line $\ell$ can be parametrized as
\begin{equation}
\label{eq:spheres-in-parabolic-pen} 
[K_t] = [\cos t \ k_1 + \sin t \ k_2] = [(\cos t, \cos t \ p + \sin t \ \xi, \sin t)].
\end{equation}
From equation \eqref{eq:1.4.4}, we see that $[K_t]$ corresponds to the oriented sphere in $S^n$ with center
\begin{equation}
\label{eq:1.5.4}
p_t = \cos t \ p + \sin t \ \xi,
\end{equation}  
and signed radius $t$.  
These are precisely the spheres through $p$ in oriented contact with the great sphere
corresponding to $[k_2]$.  Their centers lie along the geodesic in $S^n$ with initial point $p$ and initial velocity
vector $\xi$.

If we wish to work in ${\bf R}^n$, the set of {\em Lie spheres} consists of all oriented hyperspheres, oriented hyperplanes, and point spheres in ${\bf R}^n \cup \{ \infty\}$.  As in the spherical case,
we can find a bijective correspondence between the set of all Lie spheres and the set of
points on $Q^{n+1}$, and the notions of oriented contact and parabolic pencils of
Lie spheres are defined in a natural way (see, 
for example, \cite[pp. 14--23]{Cec1}).

A {\em Lie sphere transformation} 
is a projective transformation of ${\bf P}^{n+2}$ which maps the Lie quadric $Q^{n+1}$
to itself.  In terms of the geometry of $S^n$ (or ${\bf R}^n$), 
a Lie sphere transformation maps Lie spheres to Lie spheres.
Furthermore, since a Lie sphere transformation maps lines on $Q^{n+1}$ to lines on $Q^{n+1}$, 
a Lie sphere transformation preserves oriented contact of Lie spheres  (see Pinkall \cite[p. 431]{P4}
or \cite[pp. 25--30]{Cec1}).

The group of Lie sphere transformations is
isomorphic to $O(n+1,2)/ \{ \pm I \}$, where $O(n+1,2)$
is the orthogonal group for the metric in equation (\ref{Lie-metric}).
A Lie sphere transformation that maps point spheres to point spheres is a {\em M\"{o}bius transformation}, i.e.,
it is induced by a conformal diffeomorphism of $S^n$, and the set of all M\"{o}bius transformations is a subgroup of the Lie sphere group.  
An example of a Lie sphere transformation that is not
a M\"{o}bius transformation is a parallel transformations $P_t$, which fixes the center of each Lie
sphere but adds $t$ to its signed radius (see \cite[pp. 25--49]{Cec1}).

The $(2n-1)$-dimensional manifold $\Lambda^{2n-1}$ of projective lines on the quadric $Q^{n+1}$ has a 
contact structure, i.e., a
$1$-form $\omega$ such that $\omega \wedge (d\omega)^{n-1}$ does not vanish on $\Lambda^{2n-1}$.  The condition $\omega = 0$ defines a codimension one distribution $D$ on $\Lambda^{2n-1}$ which has 
integral submanifolds
of dimension $n-1$, but none of higher dimension.  Such an integral  submanifold 
$\lambda: M^{n-1} \rightarrow \Lambda^{2n-1}$ of $D$ of dimension $n-1$ is called a 
{\em Legendre submanifold} (see \cite[pp. 51--64]{Cec1}).

An oriented hypersurface $f:M^{n-1} \rightarrow S^n$ with field of unit
normals $\xi :M^{n-1} \rightarrow S^n$ naturally induces
a Legendre submanifold $\lambda = [k_1, k_2]$, where 
\begin{equation}
k_1 = (1,f,0), \quad k_2 = (0,\xi ,1),
\end{equation}
in homogeneous coordinates.
For each $x \in M^{n-1}, [k_1(x)]$ is the unique point sphere and
$[k_2 (x)]$ is the unique great sphere in the parabolic pencil of spheres in $S^n$ corresponding to 
the points on the line $\lambda (x)$.
Similarly, an immersed submanifold $\phi :V \rightarrow S^n$ of codimension 
greater than one induces a Legendre submanifold $\lambda$
whose domain is the bundle $B^{n-1}$ of unit normal
vectors to $\phi (V)$.  In each case, $\lambda$ is called the {\em Legendre lift} of the 
corresponding submanifold in $S^n$.
In a similar way, a submanifold of  ${\bf R}^n$ naturally induces a Legendre lift.

If $\beta$  is a Lie sphere transformation, then $\beta$ maps lines on
$Q^{n+1}$ to lines on $Q^{n+1}$, so it naturally
induces a map $\widetilde{\beta }$ from 
$\Lambda ^{2n-1}$ to itself.  If $\lambda$ is a
Legendre submanifold, then one can show that $\widetilde{\beta}\lambda$
is also a Legendre submanifold, which is denoted as
$\beta \lambda$ for short.  These two Legendre
submanifolds are said to be {\em Lie equivalent}.
We will also say that two submanifolds of $S^n$ or ${\bf R}^n$
are Lie equivalent, if their corresponding Legendre lifts
are Lie equivalent.

\section{Lie Invariance  of Tautness}
\label{sec:4}

Let $\phi:V \rightarrow S^n$ be an embedding of a compact, connected manifold $V$ into $S^n$, where 
$S^n$ is the unit sphere centered at the origin in Euclidean space ${\bf R}^{n+1}$.
The key idea of \'{A}lvarez Paiva \cite{Alv} 
is to formulate tautness in terms of real-valued
functions on $S^n$ whose level sets form a parabolic pencil 
of unoriented spheres in $S^n$, and then show that this is equivalent to the usual formulation 
of tautness in terms of spherical distance functions,
whose level sets in $S^n$ form a pencil of unoriented concentric spheres.  
The presentation in these notes is very similar to that in the book \cite[pp. 82--95]{Cec1}, 
and some passages are taken directly from that book.

Recall that the parabolic pencil of oriented spheres determined by the contact element $(p,\xi) \in T_1S^n$ consists of the 
oriented spheres that are all in oriented contact at the contact element $(p,\xi)$.  
As in equation \eqref{eq:spheres-in-parabolic-pen}, these spheres are represented by the points
\begin{equation}
\label{eq:parabolic-pencil}
(\cos r , \cos r\  p + \sin r\  \xi, \sin r),
\end{equation}
where $ - \pi < r < \pi$.  
This sphere has center
\begin{equation}
\label{eq:3.6.4}
q = \cos r\  p + \sin r\  \xi,
\end{equation}
which lies along the geodesic with initial point $p$ and initial velocity vector $\xi$, and it has signed radius $r$. 

To get the corresponding parabolic pencil of unoriented spheres, we restrict the value of $r$
to the interval $(0,\pi)$.  In that case,
\begin{displaymath}
r = d(p,q) = \cos^{-1} (p \cdot q),
\end{displaymath}
the spherical distance between $p$ and $q$, and so $r$ is the radius of the unoriented sphere in the pencil containing $p$ with center $q$.

Given a contact element $(p,\xi) \in T_1S^n$, we want to define a function
\begin{displaymath}
r_{(p,\xi)}: S^n - \{ p\} \rightarrow (0, \pi),
\end{displaymath}
whose level sets are unoriented spheres in the parabolic pencil
of unoriented spheres determined by $(p,\xi)$.
(We will often denote $r_{(p,\xi)}$ simply by $r$ when the context is clear.)  

As noted above, an unoriented sphere in the parabolic pencil must have its center at a point $q = \cos r\  p + \sin r\  \xi$
along the geodesic with initial point $p$ and initial velocity vector $\xi$.  Since the sphere must contain the 
point $p$, the radius of the sphere must equal 
\begin{equation}
\label{eq:unsigned-radius-r}
d(p,q) = \cos^{-1} (p \cdot q) = \cos ^{-1} (\cos r) = r,
\end{equation}
where $0 < r < \pi$.

A point $x$ in $S^n - \{p\}$ lies on this sphere precisely when 
\begin{displaymath}
d(x,q) = d(p,q) = r. 
\end{displaymath} 
This gives the condition
\begin{equation}
\label{eq:unsigned-radius-r-calculation}
\cos^{-1} (x \cdot q) =  \cos^{-1} (p \cdot q) = r,
\end{equation}
 and thus by taking the cosine of both sides of the equation, we get
\begin{equation}
\label{eq:unsigned-radius-r-implicit}
\cos r = x \cdot q = x \cdot (\cos r\  p + \sin r\  \xi).
\end{equation}

So the point $x$ in
$S^n - \{p\}$ lies on precisely one sphere $S_x$ in the parabolic pencil of unoriented spheres, as the spherical radius 
$r$ of the spheres in
the pencil varies from 0 to $\pi$.  The radius $r_{(p,\xi)}(x)$ of $S_x$ is defined implicitly by the equation 
\eqref{eq:unsigned-radius-r-implicit} above,
\begin{equation}
\label{eq:3.6.3}
\cos r = x \cdot (\cos r\  p + \sin r\  \xi).
\end{equation}
This defines a smooth function 
\begin{displaymath}
r_{(p,\xi)}: S^n - \{p\} \rightarrow (0, \pi).
\end{displaymath}
The level sets of the function $r_{(p,\xi)}$  are the unoriented spheres in the parabolic pencil of unoriented spheres
determined by ${(p,\xi)}$.
Some sample values of the function $r_{(p, \xi)}$ are
\begin{displaymath}
r_{(p, \xi)}(\xi) = \pi/4,\quad r_{(p, \xi)}(-p) = \pi/2, \quad r_{(p, \xi)}(-\xi) = 3\pi/4.
\end{displaymath}
Note that the contact element $(p,- \xi)$ determines the same pencil of unoriented spheres, and
the function $r_{(p,- \xi)} = \pi - r_{(p,\xi)}$. (See \cite[pp. 84--88]{Cec1} for diagrams related to the function 
$r_{(p, \xi)}$ and the sphere $S_x$.)

In this section, we are dealing with immersions $\phi: V \rightarrow S^n$, where $V$ is a $k$-dimensional manifold
with $k < n$.  If $x \in V$, we say that a hypersphere $S_x$ and $\phi (V)$ are {\em tangent} at $\phi (x)$ if 
\begin{displaymath}
d\phi (T_x V) \subset T_{\phi(x)} S_x,
\end{displaymath}
where $d\phi$ is the differential of $\phi$.  

A {\em curvature sphere} of $\phi(V)$ at $x$ is
a hypersphere $S_x$ that is tangent to $\phi (V)$ at $\phi (x)$, and
such that the center of $S_x$ is a focal point of $(V,x)$
(see, for example, \cite[pp. 11--20]{CR8}).

The following lemma (see \cite[pp. 84--88]{Cec1})
describes the critical point behavior of a function of the form $r_{(p,\xi)}$ on an immersed
submanifold 
\begin{displaymath}
\phi: V \rightarrow S^n.
\end{displaymath}   
The proof is a long, but straightforward, calculation
similar to the proof of the well-known Index Theorem for distance functions 
on submanifolds of Euclidean space (see, for example, Milnor \cite[pp. 32--38]{Mil}).

\begin{lemma}
\label{lem:3.6.3} 
Let $\phi:V \rightarrow S^n$ be an immersion of a connected manifold $V$ with $\dim V < n$ into $S^n$, 
and let $(p, \xi) \in T_1S^n$
such that $p \notin \phi (V)$.\\ 
{\rm (a)} A point $x_0 \in V$ is a critical point of the function $r_{(p,\xi)}$ if and only if the sphere
$S_{x_0}$ containing $\phi(x_0)$ in the parabolic pencil of unoriented spheres
determined by $(p,\xi)$ and the submanifold $\phi (V)$ are
tangent at $\phi (x_0)$.\\
{\rm (b)} If $r_{(p,\xi)}$ has a critical point at $x_0 \in V$, then this critical point is degenerate if and only
if the sphere $S_{x_0}$ is a curvature sphere of $\phi (V)$ at $x_0$.
\end{lemma}
\begin{proof}
(a) For any point $x_0 \in V$, there exists a sufficiently small neighborhood $U$ of $x_0$ such that the
restriction of $\phi$ to $U$ is an embedding. We will identify $U$ with its embedded image
$\phi (U) \subset S^n$ and omit the reference to the embedding $\phi$. For brevity, we will denote the
function $r_{(p,\xi)}$ simply as $r$.

Let $X$ be a
smooth vector field tangent to $V$ in such a neighborhood $U$ of $x_0$.  Then on $U$ we compute the derivative
\begin{equation}
\label{eq:3.6.5}
X(\cos r) = - \sin r \ X(r).
\end{equation}
Since $\sin r \neq 0$ for $r \in (0, \pi)$, we see that the functions $\cos r$ and $r$ have the same critical
points on $U$.  Using this and equation \eqref{eq:3.6.3}, we see that
\begin{equation}
\label{eq:3.6.6}
X(r) = 0 \Leftrightarrow X(\cos r) = 0 \Leftrightarrow X(x \cdot (\cos r\  p + \sin r\  \xi)) = 0.
\end{equation}
Then we compute
\begin{equation}
\label{eq:3.6.7}
X(x \cdot (\cos r p + \sin r \xi)) = X \cdot (\cos r p + \sin r \xi) + x \cdot (X(\cos r) p + X(\sin r) \xi).
\end{equation}
If $X(r) = 0$, then $X(\cos r) = X(\sin r) = 0$, and we get from the equation \eqref{eq:3.6.7} above that
\begin{equation}
\label{eq:3.6.8}
X \cdot (\cos r\  p + \sin r\  \xi) = 0.
\end{equation}
Thus, if the function $r = r_{(p,\xi)}$ has a critical point at $x_0$, we have
\begin{equation}
\label{eq:3.6.9}
X \cdot q = 0,
\end{equation}
for all $X \in T_{x_0}V$, where $ q = \cos r\  p + \sin r\  \xi$ is the center of the sphere $S_{x_0}$ in the parabolic
pencil of spheres determined by $(p, \xi)$ that contains the point $x_0$.  The normal space 
to the sphere $S_{x_0}$ in
${\bf R}^{n+1}$ at the point $x_0$
is spanned by the vectors $x_0$ and $q$.  So equation \eqref{eq:3.6.9}, together with the fact that
$X \cdot x_0 = 0$ for $X \in T_{x_0}V$, implies that the tangent space $T_{x_0}V$ is contained in the tangent
space $T_{x_0}S_{x_0}$, i.e., the sphere $S_{x_0}$ and the submanifold $\phi (V)$ are tangent at the point $\phi (x_0)$.

Conversely, if the tangent space $T_{x_0}V$ is contained in the tangent
space $T_{x_0}S_{x_0}$, then $X(r) = 0$ for all $X \in T_{x_0}V$, because $S_{x_0} - \{p\}$ is a level set of the function
$r$ in $S^n$. Thus, $r$ has a critical point at $x_0$.\\

\noindent 
(b) We now want to compute the 
Hessian of the function $r$ at a critical point $x_0 \in V$.  Let $X$ and $Y$ be
smooth vector fields tangent to $V$ on the neighborhood $U$ of $x_0$ in $V$ that was used in part (a).  
Then the value $H(X,Y)$ of 
the Hessian of $r$
at $x_0$ equals $YX(r)$ at $x_0$.  We first note that
\begin{equation}
\label{eq:3.6.10}
YX(\cos r) = Y(-\sin r \ X(r)) = Y(-\sin r) X(r) - \sin r \ YX(r).
\end{equation}
At the critical point $x_0$, we have $X(r) = 0$, and thus
\begin{equation}
\label{eq:3.6.11}
YX(\cos r) = - \sin r \ YX(r).
\end{equation}
On the other hand, by differentiating equation \eqref{eq:3.6.3} we get
\begin{equation}
\label{eq:3.6.12}
X(\cos r) = X \cdot (\cos r\  p + \sin r\  \xi) + x \cdot (X(\cos r) p + X(\sin r) \xi).
\end{equation}
Then by differentiating equation \eqref{eq:3.6.12}, we have
\begin{eqnarray}
\label{eq:3.6.13}
YX(\cos r) = D_Y X \cdot (\cos r\  p + \sin r\  \xi) + X \cdot (Y(\cos r) p + Y(\sin r) \xi) \\ 
 + \ Y \cdot (X(\cos r) p + X(\sin r) \xi) + x \cdot (YX(\cos r) p + YX(\sin r) \xi), \nonumber
\end{eqnarray}
where $D$ is the Euclidean covariant derivative on ${\bf R}^{n+1}$.  At the critical point $x_0$, we have
\begin{equation}
\label{eq:3.6.14}
X(\cos r) = X(\sin r) = Y(\cos r) = Y(\sin r) = 0,
\end{equation}
and so equation \eqref{eq:3.6.13} becomes
\begin{equation}
\label{eq:3.6.15}
YX(\cos r) = D_Y X \cdot (\cos r\  p + \sin r\  \xi) + x_0 \cdot (YX(\cos r) p + YX(\sin r) \xi).
\end{equation}
We now examine the two terms on the right side of equation \eqref{eq:3.6.15} separately.  Note that at the critical point
$x_0$, the vector $q = \cos r \ p + \sin r \ \xi$ lies in the normal
space to $\phi (V)$ at $x_0$, and so
\begin{equation}
\label{eq:3.6.16}
q = \cos r\  x_0 + \sin r\  N, 
\end{equation}
where $N$ is a unit normal vector to $\phi (V)$ at $x_0$.  Thus we can write the covariant derivative
$D_Y X$ as
\begin{equation}
\label{eq:3.6.17}
D_Y X = \nabla_Y X + \alpha (X,Y) - (X \cdot Y) x_0, 
\end{equation}
where $\nabla$ is the Levi--Civita connection on $\phi (V)$ induced from $D$, $\alpha$ is the second fundamental
form of $\phi (V)$ in $S^n$, and the term $- (X \cdot Y) x_0$ is the second fundamental
form of the sphere $S^n$ as a hypersurface in ${\bf R}^{n+1}$. 
Then using equation \eqref{eq:3.6.16} we obtain
\begin{eqnarray}
\label{eq:3.6.18}
D_Y X \cdot q &=& (\nabla_Y X + \alpha (X,Y) - (X \cdot Y) x_0) \cdot q \\
 &=& \alpha (X,Y) \cdot (\cos r\  x_0 + \sin r\  N) - (X \cdot Y) x_0 \cdot (\cos r\  x_0 + \sin r\  N) \nonumber \\
 &=& \sin r \ A_N X \cdot Y - \cos r \ X \cdot Y, \nonumber
\end{eqnarray}
since $\nabla_Y X$ is orthogonal to $q$, $\alpha (X,Y) \cdot x_0 = 0$, $N \cdot x_0 = 0$,  and 
\begin{displaymath}
\alpha (X,Y)\cdot N = A_N X \cdot Y,
\end{displaymath}
where $A_N$ is the shape operator
determined by the unit normal $N$ to $\phi (V)$ at $x_0$.

Next we consider the second term on the right side of equation \eqref{eq:3.6.15},
\begin{equation}
\label{eq:3.6.19}
x_0 \cdot (YX(\cos r) p + YX(\sin r) \xi).
\end{equation}
Note that at the critical point $x_0$ we have $X(r) = Y(r) = 0$, and so
\begin{eqnarray}
\label{eq:3.6.20}
YX(\cos r) &=& Y( - \sin r \ X(r)) = Y(- \sin r) X(r) - \sin r \ YX(r)\\ \nonumber 
&=& - \sin r \ YX(r),
\end{eqnarray}
and similarly
\begin{equation}
\label{eq:3.6.21}
YX(\sin r) = \cos r \ YX(r).
\end{equation}
Thus we have
\begin{eqnarray}
\label{eq:3.6.22}
YX (\cos r) p + YX(\sin r) \xi &=& - \sin r \ YX(r) p + \cos r \ YX(r) \xi \\
 &=& YX(r) ( -\sin r\  p + \cos r\  \xi). \nonumber 
\end{eqnarray}
So the term in equation \eqref{eq:3.6.19} above is
\begin{equation}
\label{eq:3.6.23}
x_0 \cdot (YX(\cos r) p + YX(\sin r) \xi) = YX(r) (x_0 \cdot ( -\sin r\  p + \cos r\  \xi)).
\end{equation}
From equations \eqref{eq:3.6.15}, \eqref{eq:3.6.18} and \eqref{eq:3.6.23}, we have
\begin{equation}
\label{eq:3.6.24}
YX(\cos r) = \sin r A_N X \cdot Y - \cos r X \cdot Y + YX(r) (x_0 \cdot ( -\sin r\  p + \cos r\  \xi)).
\end{equation}
Using equations \eqref{eq:3.6.11} and \eqref{eq:3.6.24}, we get
\begin{equation}
\label{eq:3.6.25}
(- \sin r - (x_0 \cdot ( -\sin r\  p + \cos r\  \xi))) YX(r) = \sin r A_N X \cdot Y - \cos r X \cdot Y.
\end{equation}
Denote the coefficient of $YX(r)$ in equation \eqref{eq:3.6.25} by 
\begin{equation}
\label{eq:3.6.26}
C = - \sin r - (x_0 \cdot ( -\sin r\  p + \cos r\  \xi)).
\end{equation}
Note that the vector 
\begin{equation}
\label{eq:3.6.27}
v =  -\sin r\  p + \cos r\  \xi
\end{equation}
is tangent at the point $q$ to the geodesic in $S^n$ from $p$ with initial tangent vector $\xi$.

On the sphere
$S_{x_0}$ centered at $q$, the function $x \cdot v$ takes its minimum value at $p$ where it is equal to $- \sin r$.  Thus
\begin{equation}
\label{eq:3.6.28}
x_0 \cdot ( -\sin r\  p + \cos r\  \xi) > - \sin r,
\end{equation}
since $x_0 \in S_{x_0}$ and $x_0 \neq p$.  So the term
\begin{equation}
\label{eq:3.6.29}
\sin r + x_0 \cdot ( -\sin r\  p + \cos r\  \xi) > 0,
\end{equation}
and the coefficient $C$ in equation \eqref{eq:3.6.26} is negative.  Thus we have from equation \eqref{eq:3.6.25},
\begin{eqnarray}
\label{eq:3.6.30}
YX(r) &=& (1/C) (\sin r \ (A_N X \cdot Y) - \cos r \ (X \cdot Y))\\
 &=& (1/C) (\sin r \ A_N - \cos r \ I) X \cdot Y.\nonumber
\end{eqnarray}
Thus the Hessian $H(X,Y)$ of $r$ at $x_0$ is degenerate if and only if there is a nonzero vector $X \in T_{x_0} V$
such that 
\begin{equation}
\label{eq:3.6.31}
(\sin r \ A_N - \cos r \ I) X = 0,
\end{equation}
that is,
\begin{equation}
\label{eq:3.6.32}
A_N X = \cot r \ X.
\end{equation}
This equation holds if and only if
$\cot r$ is an eigenvalue of $A_N$ with corresponding principal vector $X$, i.e., the point
$q = \cos r\  x_0 + \sin r\  N$ is a focal point of $\phi (V)$ at $x_0$, and thus the corresponding sphere $S_{x_0}$ is
a curvature sphere of $\phi (V)$ at $x_0$. 
\end{proof}

Next we show that except for $(p,\xi)$ in a set of measure zero in $T_1S^n$, the function $r_{(p,\xi)}$ is a Morse function
on $\phi (V)$.  This is accomplished by using Sard's Theorem 
in a manner similar to the usual proof that
a generic distance 
function is a Morse function on $\phi (V)$ (see, for example, Milnor \cite[pp. 32--38]{Mil}).
More specifically, from Lemma \ref{lem:3.6.3} we know that the function $r_{(p,\xi)}$,
for $p \notin \phi(V)$, is a Morse function
on $\phi (V)$ unless the parabolic pencil of unoriented spheres 
determined by $(p,\xi)$ contains a curvature
sphere of $\phi (V)$.  We now show that the set of $(p,\xi)$ in $T_1S^n$ such that the parabolic pencil determined
by $(p,\xi)$ contains a curvature sphere of $\phi (V)$ has measure zero in $T_1S^n$.

Let $B^{n-1}$ denote the unit normal
bundle of the submanifold $\phi (V)$ in $S^n$.  Note that in the case where
$\phi (V)$ is a hypersurface, $B^{n-1}$ is a two-sheeted covering of $V$.  We first recall the 
{\em normal exponential map} 
\begin{displaymath}q: B^{n-1} \times (0,\pi) \rightarrow S^n,
\end{displaymath}
defined as follows. For a point $(x,N)$ in $B^{n-1}$ and $r \in (0,\pi)$, we define
\begin{equation}
\label{eq:3.6.33}
q((x,N),r) = \cos r\  x + \sin r\  N.
\end{equation}
Next we define a $(2n-1)$-dimensional manifold $W^{2n-1}$ by
\begin{equation}
\label{eq:3.6.34}
W^{2n-1} = \{((x,N),r,\eta) \in B^{n-1} \times (0,\pi) \times S^n \mid  \eta \cdot q((x,N),r) = 0\}.
\end{equation}
The manifold $W^{2n-1}$ is a fiber bundle over $B^{n-1} \times (0,\pi)$ with fiber diffeomorphic to $S^{n-1}$.
For each point $((x,N),r) \in B^{n-1} \times (0,\pi)$, the fiber consists of all unit vectors 
$\eta$ in ${\bf R}^{n+1}$
that are tangent to $S^n$ at the point $q((x,N),r)$.

We define a map,
\begin{equation}
\label{eq:3.6.35}
F: W^{2n-1} \rightarrow T_1S^n,
\end{equation}
by
\begin{equation}
\label{eq:3.6.35a}
F((x,N),r,\eta) = (\cos r\  q + \sin r\  \eta,\   \sin r\  q - \cos r\  \eta),
\end{equation}
where $q = q((x,N),r)$ is defined in equation \eqref{eq:3.6.33}.
We will now show
that if the parabolic pencil of unoriented spheres 
determined by $(p,\xi)\in T_1S^n$ 
contains a curvature sphere of $\phi (V)$,
then $(p,\xi)$ is a critical value of $F$.  Since the set of critical values of $F$ has measure zero by 
Sard's Theorem
(see, for example, Milnor \cite[p. 33]{Mil}), this will give the desired conclusion.

The proof of the following lemma is taken from the book \cite[pp. 89--92]{Cec1}.

\begin{lemma}
\label{lem:3.6.4} 
Let $\phi:V \rightarrow S^n$ be an immersion of a connected manifold $V$ 
with $\dim V < n$ into $S^n$, and let $B^{n-1}$ be the
unit normal 
bundle of $\phi (V)$.  Define 
\begin{displaymath}
F: W^{2n-1} \rightarrow T_1S^n,
\end{displaymath}
as in equation \eqref{eq:3.6.35a}.  If the parabolic pencil of unoriented spheres determined by
$(p,\xi)$ in $T_1S^n$ contains a curvature sphere of $\phi (V)$, then $(p,\xi)$ is a critical value of $F$.
Thus, the set of such $(p,\xi)$ has measure zero in $T_1S^n$.
\end{lemma}
\begin{proof}
Suppose that $(p,\xi) \in T_1S^n$ is such that the sphere $S$ with center
\begin{equation}
\label{eq:3.6.38}
q_0 = \cos r_0\  p + \sin r_0\  \xi,
\end{equation}
and radius $r_0 \in (0,\pi)$ is a curvature sphere of $\phi (V)$ at $x_0 \in V$.  Then there is a unit normal
$N_0$ to $\phi (V)$ at $\phi (x_0)$ such that
\begin{equation}
\label{eq:3.6.39}
q_0 = \cos r_0\  x_0 + \sin r_0\  N_0,
\end{equation}
and $\cot r_0$ is an eigenvalue of the shape operator $A_{N_0}$ with corresponding nonzero principal vector $X$
such that
\begin{equation}
\label{eq:3.6.40}
A_{N_0} X = \cot r_0 \ X.
\end{equation}
(Here, as before, we suppress the notation for the immersion $\phi$ and write $x_0$ instead of
$\phi(x_0)$ in equation \eqref{eq:3.6.39}.)
Note that if we take
\begin{displaymath}
\eta_0 = \sin r_0\  p - \cos r_0\  \xi,
\end{displaymath}
then using equation \eqref{eq:3.6.38}, we get $\eta_0 \cdot q_0 = 0$, and
\begin{displaymath}
p = \cos r_0\  q_0 + \sin r_0\  \eta_0, \quad \xi = \sin r_0\  q_0 - \cos r_0\  \eta_0,
\end{displaymath}
so that from equation \eqref{eq:3.6.39}, we have
\begin{displaymath}
(p, \xi) = F ((x_0,N_0), r_0, \eta_0).
\end{displaymath}

We now want to show that $((x_0,N_0), r_0, \eta_0)$ is a critical point of $F$, and thus $(p,\xi)$ is a critical
value of $F$. To compute the differential $dF$ at $((x_0,N_0), r_0, \eta_0)$, we need to put local coordinates
on a neighborhood of this point in $W^{2n-1}$.  First, we choose local coordinates on the unit normal  bundle
$B^{n-1}$ in a neighborhood of the point $(x_0, N_0)$ in $B^{n-1}$ in the following way.  
Suppose that $V$ has dimension $k \leq n-1$.
Let $U$ be a local coordinate neighborhood of $x_0$ in $V$ with coordinates $(u_1,\ldots,u_k)$ such that
$x_0$ has coordinates $(0,\ldots,0)$.  Choose orthonormal
normal vector fields, 
\begin{displaymath}
N_0,N_1,\ldots,N_{n-1-k},
\end{displaymath}
on $U$ such that the vector field $N_0$ agrees with the given vector $N_0$ at $x_0$, and $\nabla_X^\perp N_0 = 0$
at $x_0$, where $\nabla^\perp$ is the connection in the normal 
bundle to $\phi (V)$, for $X$ as in equation \eqref{eq:3.6.40}.  If $x \in U$ and
$N$ is a unit normal vector to $\phi (V)$ at $\phi(x)$ with $N \cdot N_0 > 0$, then we can write
\begin{equation}
\label{eq:3.6.41}
N = (1 - \sum_{i=1}^{n-1-k} s_i^2)^{1/2}N_0 + s_1 N_1+ \cdots + s_{n-1-k} N_{n-1-k}, 
\quad \sum_{i=1}^{n-1-k} s_i^2 < 1. 
\end{equation}
Thus $(u_1,\ldots,u_k,s_1,\ldots,s_{n-1-k})$ are local coordinates on an open set $O$ in 
the unit normal bundle $B^{n-1}$ over the
open set $U \subset V$, and $(x_0,N_0)$ has coordinates $(0,\ldots,0)$. 
(In the case where $V$ has codimension one in $S^n$, just use the coordinates
$(u_1,\ldots,u_{n-1})$ from $U$, since $B^{n-1}$ is a 2-sheeted covering of $U$.) 
Therefore, any tangent vector to $B^{n-1}$ at a point $(x,N)$ can be written in the
form $(Y,W)$, where $Y \in T_xV$ and $W$ is a linear combination of $\{\partial /\partial s_1,\ldots,\partial /\partial 
s_{n-1-k}\}$.

Next we wish to get local coordinates on the $S^{n-1}$-fiber near the vector $\eta_0$ at $q_0$.    
Let $\{E_1,\ldots,E_n\}$ be a local orthonormal frame of tangent vectors to $S^n$ in a neighborhood of the
point $q_0$ defined in equation \eqref{eq:3.6.38} such that $E_n(q_0) = \eta_0$.  Then we define
for $((x,N),r)$ near $((x_0,N_0),r_0)$ in $B^{n-1} \times (0,\pi)$,
\begin{equation}
\label{eq:3.6.36}
\eta((x,N),r,(t_1,\ldots,t_{n-1})) = t_1 E_1 (q) + \cdots + (1 - \sum_{i=1}^{n-1} t_i^2)^{1/2} E_n (q),
\end{equation}
where $q = q((x,N),r)$ is defined in equation \eqref{eq:3.6.33}, and $\sum_{i=1}^{n-1} t_i^2 < 1$. 
Thus we have local coordinates, 
\begin{displaymath}
(u_1,\ldots,u_k,s_1,\ldots,s_{n-1-k}, r, t_1,\ldots,t_{n-1}),
\end{displaymath}
on a neighborhood
of the point $((x_0,N_0),r_0,\eta_0)$ in $W^{2n-1}$, and the point $((x_0,N_0),r_0,\eta_0)$ has coordinates
$(0,\ldots,0,r_0,0,\ldots,0)$.
 
We now want to calculate the differential $dF$ of the vector $((X,0),0,0)$ tangent to 
$W^{2n-1}$ at the point $((x_0,N_0),r_0,\eta_0)$.
We begin by computing the differential
$dq((X,0),0)$ at the point $((x_0,N_0),r_0)$, where $q$ is given by equation \eqref{eq:3.6.33}.  Let $\gamma (t)$
be a curve in $U$ with $\gamma(0) = x_0$ and initial tangent vector $\gamma'(0) = X$.  Then $(\gamma(t),N_0(\gamma(t))$
is a curve in $B^{n-1}$ with initial tangent vector $(X,0)$.  So at $((x_0,N_0),r_0)$, we have that
$dq((X,0),0)$ is the initial tangent vector to the curve
\begin{equation}
\label{eq:3.6.42}
\cos r_0\  \gamma(t) + \sin r_0\  N_0(\gamma(t)),
\end{equation}
and so
\begin{equation}
\label{eq:3.6.43}
dq((X,0),0) = \cos r_0\  X + \sin r_0\  D_X N_0.
\end{equation}
Since $X \cdot N_0 = 0$, we have $D_X N_0 = \widetilde{\nabla}_X N_0$, where $\widetilde{\nabla}$ is the Levi--Civita
connection on $S^n$, and we know that
\begin{equation}
\label{eq:3.6.44}
\widetilde{\nabla}_X N_0 = - A_{N_0} X + \nabla_X^\perp N_0.
\end{equation}
We have chosen $N_0$ so that $\nabla_X^\perp N_0 = 0$, so by equation \eqref{eq:3.6.40}, we have
\begin{equation}
\label{eq:3.6.45}
\widetilde{\nabla}_X N_0 = - A_{N_0} X = - \cot r_0\  X.
\end{equation}
Thus by equations \eqref{eq:3.6.43} and \eqref{eq:3.6.45}, we have
\begin{equation}
\label{eq:3.6.46}
dq((X,0),0) = \cos r_0\  X + \sin r_0 (- \cot r_0\  X) = 0.
\end{equation}
Next we want to compute the differential $d\eta ((X,0),0,0)$ at $((x_0,N_0),r_0,\eta_0)$, 
for $\eta$ as defined in equation \eqref{eq:3.6.36}. Note that the coordinates $(t_1,\ldots,t_{n-1})$ are
$(0,\ldots,0)$ at $((x_0,N_0),r_0,\eta_0)$, since $E_n(q_0) = \eta_0$. 

From equations
\eqref{eq:3.6.36} and \eqref{eq:3.6.46}, we see that 
\begin{equation}
\label{eq:3.6.47}
d\eta((X,0),0,0) = dE_n (dq((X,0),0)) = 0,
\end{equation}
at the point $((x_0,N_0),r_0,\eta_0)$, since $dq((X,0),0) = 0$.  
Thus from equations \eqref{eq:3.6.46} and \eqref{eq:3.6.47}, we have
\begin{eqnarray}
\label{eq:3.6.48}
dF((X,0),0,0) = ( \cos r_0\  dq((X,0),0) &+& \sin r_0\  d\eta((X,0),0,0), \\
\sin r_0\  dq((X,0),0) &-& \cos r_0\  d\eta((X,0),0,0)) = (0,0). \nonumber
\end{eqnarray} 
Thus $((x_0,N_0),r_0,\eta_0)$ is a critical point of $F$, and the contact element
$(p,\xi) = F((x_0,N_0),r_0,\eta_0)$ is a critical value of $F$.  This shows that every contact element whose
corresponding parabolic pencil contains a curvature sphere of $\phi (V)$ is a critical value
of $F$. By Sard's Theorem, 
the set of critical values of the map $F$  has measure zero in $T_1S^n$, and so the
set of contact elements $(p,\xi)$ in $T_1S^n$ whose parabolic pencil contains a curvature sphere of $\phi (V)$
has measure zero. 
\end{proof}

As a consequence of Lemmas \ref{lem:3.6.3} and \ref{lem:3.6.4}, we have the following corollary.

\begin{corollary}
\label{cor:3.6.5} 
Let $\phi:V \rightarrow S^n$ be an immersion of a connected manifold $V$ with $\dim V < n$ into
$S^n$.  For almost all
$(p,\xi) \in T_1S^n$, the function $r_{(p,\xi)}$ is a Morse function on $V$. 
\end{corollary}
\begin{proof}
By Lemma \ref{lem:3.6.3}, the function $r_{(p,\xi)}$ is a Morse function on $V$ if and only if 
$p \notin \phi (V)$ and the parabolic
pencil of unoriented spheres determined by $(p,\xi)$ does not contain a curvature sphere of $\phi(V)$.
The set of $(p,\xi)$ such that $p \in \phi (V)$ has measure zero, since $\phi(V)$ is a submanifold of codimension
at least one in $S^n$. The set of $(p,\xi)$ such that the parabolic pencil determined by $(p,\xi)$ contains
a curvature sphere of $\phi(V)$ has measure zero by Lemma \ref{lem:3.6.4}.  Thus, except for $(p,\xi)$ in the 
set of measure zero obtained by taking the union of these two sets, 
the function $r_{(p,\xi)}$ is a Morse function on $V$.
\end{proof}

We are now ready to give a definition of tautness for Legendre submanifolds in Lie sphere geometry
(see  \cite[pp. 89--92]{Cec1}).  Recall the
diffeomorphism from $T_1S^n$ to the manifold $\Lambda^{2n-1}$ of lines
on the Lie 
quadric $Q^{n+1}$ given by
equations \eqref{eq:3.1.8} and \eqref{eq:3.1.9},
\begin{equation}
\label{eq:3.6.49}
(p,\xi) \mapsto [(1,p,0),(0,\xi,1)] = \ell \in \Lambda^{2n-1}.
\end{equation}
Under this correspondence, an oriented sphere $S$ in $S^n$ belongs to the 
parabolic pencil of oriented
spheres determined by $(p,\xi) \in T_1S^n$ if and only if the point $[k]$ in $Q^{n+1}$ corresponding to $S$
lies on the line $\ell$. Thus, the parabolic pencil of oriented spheres determined by a 
contact element $(p,\xi)$ contains a
curvature sphere
$S$ of a Legendre submanifold $\lambda: B^{n-1} \rightarrow \Lambda^{2n-1}$ if and only if
the corresponding line $\ell$ contains the point $[k]$ corresponding to $S$.  We now define the notion of Lie-tautness
for compact Legendre submanifolds as follows.  Here by ``almost every,'' we mean except for a set of measure zero.

\begin{definition}
\label{def:3.6.6} 
{\rm A compact, connected Legendre submanifold 
\begin{displaymath}
\lambda: B^{n-1} \rightarrow \Lambda^{2n-1}
\end{displaymath}
is said to be
{\em Lie-taut} if for almost every line $\ell$ on the Lie quadric $Q^{n+1}$, the number of points
$x \in B^{n-1}$ such that $\lambda(x)$ intersects $\ell$ is $\beta(B^{n-1};{\bf Z}_2)/2$, i.e., one-half
the sum of the ${\bf Z}_2$-Betti numbers of $B^{n-1}$.}
\end{definition}

Equivalently, this definition says that for almost every contact element $(p,\xi)$ in $T_1S^n$, the number of points
$x \in B^{n-1}$ such that the contact element corresponding to $\lambda(x)$ is in oriented contact with some sphere
in the parabolic pencil of oriented spheres determined by $(p,\xi)$ is $\beta(B^{n-1};{\bf Z}_2)/2$.

The property of Lie-tautness is clearly invariant under Lie sphere transformations, that is, if 
$\lambda: B^{n-1} \rightarrow \Lambda^{2n-1}$ is Lie-taut and $\alpha$ is a Lie sphere transformation, then
the Legendre submanifold $\alpha \lambda: B^{n-1} \rightarrow \Lambda^{2n-1}$ is also Lie-taut.  This follows
from the fact that the line $\lambda (x)$ intersects a line $\ell$ if and only if the line 
$\alpha (\lambda (x))$ intersects the line $\alpha (\ell)$, and $\alpha$ maps the complement of a set of measure zero 
in $\Lambda^{2n-1}$ to 
the complement of a set of measure zero in $\Lambda^{2n-1}$.

\begin{remark}
\label{rem:3.6.7} 
{\rm The factor of one-half in the definition comes from the fact that Lie sphere geometry deals with oriented contact and 
not just unoriented tangency, as is made clear in the proof of Theorem \ref{thm:3.6.6} below.  
Note here that if
$\phi:V \rightarrow S^n$ is an embedding of a compact, connected manifold $V$ into $S^n$, and $B^{n-1}$ is the
unit normal bundle of $\phi (V)$, then the 
{\em Legendre lift} of $\phi$ is defined to be the Legendre submanifold
$\lambda: B^{n-1} \rightarrow \Lambda^{2n-1}$ given by
\begin{equation}
\label{eq:3.6.50}
\lambda(x,N) = [(1,\phi(x),0),(0,N,1)],
\end{equation}
where $N$ is a unit normal vector to $\phi(V)$ at $\phi(x)$.
If $V$ has dimension $n-1$, then $B^{n-1}$ is a two-sheeted covering of $V$.  If $V$ has dimension less than $n-1$,
then $B^{n-1}$ is diffeomorphic to a tube $W^{n-1}$ of sufficiently small radius over $\phi(V)$ so that $W^{n-1}$
is an embedded hypersurface in $S^n$.  In either case,
\begin{displaymath}
\beta(B^{n-1};{\bf Z}_2) = 2 \beta(V;{\bf Z}_2).
\end{displaymath}
This is obvious
in the case where $V$ has dimension $n-1$, and it was proved by Pinkall \cite{P5} in the case where $V$
has dimension less than $n-1$.} 
\end{remark}

Since Lie-tautness is invariant under Lie sphere transformations, the following theorem establishes
that tautness is Lie invariant.  Recall that a taut immersion $\phi:V \rightarrow S^n$ must in fact be an embedding
(see Carter and West \cite{CW1}).

\begin{theorem}
\label{thm:3.6.6} 
Let $\phi:V \rightarrow S^n$ be an embedding of a compact, connected manifold $V$ with $\dim V < n$ into $S^n$. Then
$\phi(V)$ is a taut submanifold in $S^n$ if and only if the Legendre lift $\lambda: B^{n-1} \rightarrow \Lambda^{2n-1}$
of $\phi$ is Lie-taut.
\end{theorem}
\begin{proof}
Suppose that $\phi(V)$ is a taut submanifold in $S^n$, and let 
\begin{displaymath}
\lambda: B^{n-1} \rightarrow \Lambda^{2n-1}
\end{displaymath}
be the Legendre lift of $\phi$.
Let $(p,\xi) \in T_1S^n$ such that $p \notin \phi(V)$
and such that the parabolic pencil of unoriented spheres 
determined by $(p,\xi)$ does not contain a curvature
sphere of $\phi(V)$.  By Lemma \ref{lem:3.6.4}, the set of such $(p,\xi)$ is the complement of a set of measure
zero in $T_1S^n$.  For such $(p,\xi)$, the function $r_{(p,\xi)}$ is a Morse function on $V$, and the sublevel set 
\begin{equation}
\label{eq:3.6.51}
V_s(r_{(p,\xi)}) = \{x \in V \mid r_{(p,\xi)}(x) \leq s\} = \phi(V) \cap B, \quad 0 < s < \pi,
\end{equation}
is the intersection of $\phi(V)$ with a closed ball $B \subset S^n$.  By tautness and Theorem~\ref{thm:3.6.2}, the map on 
${\bf Z}_2$-\v{C}ech homology,
\begin{equation}
\label{eq:3.6.52}
H_*(V_s(r_{(p,\xi)})) = H_*(\phi^{-1}(B)) \rightarrow H_*(V),
\end{equation}
is injective for every $s \in {\bf R}$, and so by Theorem~\ref{thm:3.6.1}, the function $r_{(p,\xi)}$ has
$\beta(V;{\bf Z}_2)$ critical points on $V$.  

By Lemma \ref{lem:3.6.3}, a point $x \in V$ is a critical point of $r_{(p,\xi)}$ if and only if
the unoriented sphere $S_x$ in the parabolic pencil determined
by $(p,\xi)$ containing $x$ is tangent to $\phi(V)$ at $\phi(x)$. At each such point $x$, exactly one contact
element $(x,N) \in B^{n-1}$ is in oriented contact with the oriented sphere $\widetilde{S}_x$ through $x$
in the parabolic pencil of oriented spheres determined by $(p,\xi)$.  Thus, the number of critical
points of $r_{(p,\xi)}$ on $V$ equals the number of points $(x,N) \in B^{n-1}$ such that $(x,N)$ is in oriented contact
with an oriented sphere in the parabolic pencil of oriented spheres
determined by $(p,\xi)$. 

Thus there are 
\begin{displaymath}
\beta(V;{\bf Z}_2) = \beta(B^{n-1};{\bf Z}_2)/2
\end{displaymath}
points $(x,N) \in B^{n-1}$ such that $(x,N)$ is in oriented contact
with an oriented sphere in the parabolic pencil of oriented spheres
determined by $(p,\xi)$.
This means that there
are $\beta(B^{n-1};{\bf Z}_2)/2$ points $(x,N) \in B^{n-1}$ such that the line $\lambda(x,N)$ 
intersects the line
$\ell$ on $Q^{n+1}$ corresponding to the contact element $(p,\xi)$.
Since this true for almost every $(p,\xi) \in T_1S^n$, the 
Legendre lift $\lambda$ of $\phi$ is Lie-taut.

Conversely, suppose that the Legendre lift $\lambda: B^{n-1} \rightarrow \Lambda^{2n-1}$ of $\phi$ is Lie-taut.
Then for all $(p,\xi) \in T_1S^n$ except for a set $Z$ of measure zero, the number of points
$(x,N) \in B^{n-1}$ that are in oriented contact with some sphere
in the parabolic pencil of oriented spheres determined by $(p,\xi)$ is $\beta(B^{n-1};{\bf Z}_2)/2 = \beta(V;{\bf Z}_2)$.
This means that the corresponding
function $r_{(p,\xi)}$ has $\beta(V;{\bf Z}_2)$ critical points on $V$.
By Theorem~\ref{thm:3.6.1}, this implies that for a closed ball $B \subset S^n$ such that $\phi^{-1}(B) = V_s(r_{(p,\xi)})$ 
for $(p,\xi) \notin Z$ and $s \in {\bf R}$, the map on homology,
\begin{equation}
\label{eq:3.6.53}
H_*(\phi^{-1}(B)) \rightarrow H_*(V),
\end{equation}
is injective. On the other hand, 
if $B$ is a closed ball corresponding to a sublevel set of $r_{(p,\xi)}$ for $(p,\xi) \in Z$,
then since $Z$ has measure zero, one can produce a nested sequence,
\begin{displaymath}
\{B_i\}, \quad i= 1,2,3,\ldots,
\end{displaymath}
of closed balls (coming from $r_{(p,\xi)}$ for 
$(p,\xi) \notin Z$) satisfying 
\begin{equation}
\label{eq:3.6.54}
\phi^{-1}(B_i) \supset \phi^{-1}(B_{i+1}) \supset \cdots \supset \cap_{j=1}^\infty \phi^{-1}(B_j) = \phi^{-1}(B),
\end{equation}
for $i= 1,2,3,\ldots,$ such that the homomorphism in ${\bf Z}_2$-homology,
\begin{equation}
\label{eq:3.6.55}
H_*(\phi^{-1}(B_i)) \rightarrow H_*(V),\ {\mbox{\rm is injective for }}i= 1,2,3,\ldots
\end{equation}
If equations \eqref{eq:3.6.54} and \eqref{eq:3.6.55} are satisfied, then the map
\begin{equation}
\label{eq:3.6.55a}
H_*(\phi^{-1}(B_i)) \rightarrow H_*(\phi^{-1}(B_j))\ {\mbox{\rm is injective for all }}i > j.
\end{equation}
The continuity property of \v{C}ech homology
(see Eilenberg--Steenrod \cite[p. 261]{EiS}) says that
\begin{displaymath}
H_*(\phi^{-1}(B)) = \stackrel{\leftarrow}{\lim_{i \rightarrow \infty}} H_*(\phi^{-1}(B_i)).
\end{displaymath}
Equation \eqref{eq:3.6.55a} and Theorem 3.4 of Eilenberg--Steenrod \cite[p. 216]{EiS} on 
inverse limits imply that the map
\begin{displaymath}
H_*(\phi^{-1}(B)) \rightarrow H_*(\phi^{-1}(B_i))
\end{displaymath}
is injective for each $i$.  Thus, from equation \eqref{eq:3.6.55}, we get that the map
\begin{displaymath}
H_*(\phi^{-1}(B)) \rightarrow H_*(V)
\end{displaymath}
is also injective.  Since this holds for all closed balls $B$ in $S^n$, the embedding $\phi(V)$ is taut by
Theorem \ref{thm:3.6.2}. (See Kuiper \cite{Ku3} or
Cecil--Ryan \cite[Theorem 5.4, pp. 25--26]{CR7} for an example
of this type of \v{C}ech homology argument.)
\end{proof}

\noindent
Another formulation of the Lie invariance of tautness is the following corollary.

\begin{corollary}
\label{cor:3.6.7} 
Let $\phi:V \rightarrow S^n$ and $\psi:V \rightarrow S^n$ be two embeddings of a compact, connected  
manifold $V$ with $\dim V < n$ into $S^n$, 
such that their corresponding Legendre lifts are Lie equivalent.
Then $\phi$ is a taut embedding  if and only if $\psi$ is a taut embedding.
\end{corollary}
\begin{proof}
Since the Legendre lifts of $\phi$ and $\psi$ are Lie equivalent, the unit normal bundles of $\phi (V)$ and 
$\psi (V)$ must be diffeomorphic, and we will denote them both by $B^{n-1}$. Now 
let $\lambda: B^{n-1} \rightarrow \Lambda^{2n-1}$ and $\mu: B^{n-1} \rightarrow \Lambda^{2n-1}$ be the Legendre
lifts of $\phi$ and $\psi$, respectively.  By Theorem \ref{thm:3.6.6}, $\phi$ is taut if and only if $\lambda$ is
Lie-taut, and $\psi$ is taut if and only if $\mu$ is Lie-taut.  Further, since $\lambda$ and $\mu$ are Lie equivalent,
$\lambda$ is Lie-taut if and only if $\mu$ is Lie-taut, so it follows that $\phi$ is taut if and only if $\psi$ is taut.
\end{proof}

\noindent Department of Mathematics and Computer Science

\noindent College of the Holy Cross

\noindent Worcester, MA 01610, U.S.A.

\noindent email: tcecil@holycross.edu

\end{document}